\documentclass[12pt]{amsart}

\usepackage{geometry}
\usepackage{amssymb,amsfonts,amsthm,amsmath,amscd}
\usepackage{lscape}
\usepackage{multirow}
\usepackage{fancyvrb}
\usepackage{enumerate}

\theoremstyle{plain}
\newtheorem{theorem}{Theorem}
\newtheorem{lemma}{Lemma}

\newtheorem{conjecture}{Conjecture}

\theoremstyle{definition}
\newtheorem{definition}{Definition}
\newtheorem{example}{Example}
\newtheorem{remark}{Remark}

\def\S{\mathcal{S}}
\def\G{\mathcal{G}}
\def\A{\mathcal{A}}

\begin{document}

\title{On the expansion of three-element subtraction sets}

\author{Nhan Bao Ho}
\address{Department of Mathematics, La Trobe University, Melbourne, Australia 3086}
\email{nhan.ho@latrobe.edu.au, nhanbaoho@gmail.com}

%


\thanks{}

\begin{abstract}
We study the periodicity of nim-sequences for subtraction games having subtraction sets with three elements. In particular, we give solutions in several cases, and we describe how these subtraction sets can be augmented by additional numbers without changing the nim-sequences. The paper concludes with a conjecture on ultimately bipartite subtraction games.
\end{abstract}

\maketitle

\section{Introduction}

A \emph{subtraction game} is a two-player game involving a pile of coins and a finite set $S$ of positive integers called the \emph{subtraction set}. The two players move alternately, subtracting some $s$ coins such that $s \in S$. The player who makes the last move wins. Subtraction games provide classical examples of impartial combinatorial games; see \cite{les, ww1, ww3}. They are completely understood for two-element subtraction sets. For larger subtraction sets, they are known to be all ultimately periodic \cite[p.38]{guy}, but a complete solution of these games is still not known, even for three-element subtraction sets. The purpose of this paper is to report further results in this area.

Throughout this paper, the subtraction set $S=\{s_1, s_2, \ldots, s_k\}$ will be ordered  $s_1 < s_2 < \cdots < s_k$. The subtraction game with subtraction set $S$ is denoted by $\S$. When we need to specify the subtraction set, we will use $\S(S)$ or $\S(s_1, s_2, \ldots, s_k)$.

For each nonnegative integer $n$, denote $\G(n)$ the \emph{Sprague-Grundy value}, or \emph{nim-value} for short, of the single pile of size $n$ of the subtraction game $\S$. The sequence $\{\G(n)\}_{n \geq 0}$ is called the \emph{nim-sequence}.

\begin{sloppypar}
Suppose for the moment that a subtraction set $\{s_1, s_2, \ldots, s_k\}$ has $d = \gcd(s_1, s_2, \ldots, s_k) > 1$. Let $s_i' = s_i/d$ for $1 \leq i \leq k$. The nim-sequence for the game $\S(s_1, s_2, \ldots, s_k)$ is exactly the $d$-plicate of that for the game $\S(s_1', s_2', \ldots, s_k')$. That means the former can be obtained from the latter by repeating each value of the latter exactly $d$ times \cite[p.~529]{ww3}. Thus, it suffices to consider subtraction sets whose members are relatively prime.
\end{sloppypar}

Recall that the sequence $\{\G(n)\}_{n \geq 0}$ is said to be \emph{ultimately periodic} if there exist integers $p \geq 1$ and $n_0 \geq 1$ such that $\G(n + p) = \G(n)$ for all $n \geq n_0$. The smallest such numbers $n_0$ and $p$ are called the \emph{pre-period length} and \emph{period length} respectively \cite[p.~145]{les}.  If $n_0 = 0$, the sequence is said to be \emph{purely periodic}. A purely periodic (resp.~ultimately periodic) game is called {\it bipartite} (resp.~{\it ultimately bipartite}) if $p = 2$. (Ultimately bipartite subtraction games ultimately have alternating nim-values $0,1,0,1,\ldots$). The periodicity of subtraction games is discussed in \cite{les, Ingo, ww1, ww3}. Bipartite games are first introduced in \cite{bi}.

Throughout this paper, when saying that a subtraction game has periodic nim-values $g_1 g_2\ldots g_p$, we mean that ultimately, the nim-sequence is the infinite repetition of the subsequence $g_1,g_2,\ldots,g_p$. The $n_0$ nim-values $\G(0),\G(1), \ldots, \G(n_0-1)$ is called the pre-periodic nim-values. The  omission of commas between the nim-values does not lead to any misunderstanding as all nim-values presented in this paper have exactly one digit.

The following lemma gives us a useful idea of how much calculation we need to perform to determine the pre-period and period lengths.

\begin{lemma} \label{per.sub} \cite[p.~148]{les}
Let $s_k = \max(S)$. For minimal $n_0$ and $p$ such that $\G(n+p)=\G(n)$ for $n_0 \leq n < n_0+s_k$, the subtraction game $\S$ is purely periodic with the pre-period length $n_0$ and the period length $p$.
\end{lemma}

\begin{remark} \label{R.exp} \cite[p. 84]{ww1}
For a given subtraction set $S$, if there exists a positive integer $s$ such that $\G(n+s) \neq \G(n)$ for all nonnegative integers $n$, then $s$ can be adjoined to the subtraction set $S$ without changing the nim-sequence. In this case, for brevity, we will simply say that $s$ {\it can be adjoined to} $S$. The set of all such elements, including elements in $S$, is called the {\it expansion set} of $S$ and denoted by $S^{ex}$.
\end{remark}

In reality, to identify such an $s$, we need to calculate only a small range of nim-values, as much as the calculation needed in Lemma \ref{per.sub}. We detail this range as follow.

\begin{theorem} \label{ext.form}
Let $\S$ be a subtraction game with pre-period length $n_0$ and period length $p$.
\begin{enumerate}  \itemsep0em
\item [$(i)$] A number $s < n_0+p$ can be adjoined to $S$ if and only if $\G(n+s) \neq \G(n)$ for all $n$ such that $0 \leq n < n_0+p$.
\item [$(ii)$] A number $s \geq n_0 + p$ can be adjoined to $S$ if and only if $s-p$ can be.
\end{enumerate}
\end{theorem}
\begin{proof}
The theorem follows immediately from  Remark \ref{R.exp} and the definition of ultimately periodicity.
\end{proof}

It follows from Theorem \ref{ext.form} that if $s \geq n_0 + p$ and $s$ can be adjoined to $S$ then $s = s'+mp$ for some $m$ and $s'$ such that $n_0 \leq s' < n_0+p$ and $s'$ can be adjoined to $S$. We are going to employ this idea to represent $S^{ex}$ with no more than $\max(n_0 + p, s_k)$ elements with $s_k = \max(S)$. (There is another method to present $S^{ex}$ with no more than $n_0+p$ elements but we avoid its complexity).

Let us partition $S$ into two parts: $S_1$ containing elements $s$ of $S$ such that $s < n_0$ and $S_2 = S \setminus S_1$. Let $S_1'$ (resp.~$S_2'$) be the set of all $s$ which can be adjoined to $S$ such that $s \notin S$ and $s < n_0$ (resp.~$n_0 \leq s < n_0 + p$). Then
\[S^{ex} = \{S_1 \cup S_1'\} \cup \{S_2 \cup S_2'\}^{\ast p}\]
in which $T^{\ast p} = \{t+mp| t \in T, m \geq 0\}$. Note that in the formula for $S^{ex}$, some of the sets $S_1$, $S_1'$, $S_2$, and $S_2'$ may be empty. In particular, $S_1 \cup S_1' = \emptyset$ if $\S$ is purely periodic ($n_0 = 0$).

\begin{example}
The subtraction game $\S(1, 8, 11, 27)$ is ultimately periodic with $n_0 = 13$, $p = 19$ and has expansion set $S^{ex} = \{1, 8, 11\} \cup \{13, 20, 27\}^{\ast 19}$. Here $S_1 = \{1,8,11\}$, $S_1' = \emptyset$, $S_2 = \{27\}$, and $S_2' = \{13,20\}$.
\end{example}

\begin{definition}
If $S^{ex} = \{S_1\} \cup \{S_2\}^{\ast p}$ (equivalently, $S_1' \cup S_2' = \emptyset$), the set $S$ is said to be \emph{non-expandable}. Otherwise, it is {\it expandable}.
\end{definition}

Some values of the expansion sets of certain subtraction sets with numbers up to 7 can be found in \cite[pp.~84-85]{ww1}.

The following result from \cite{bi} shows that for bipartite games, $S^{ex}$ is the set of odd positive integers.


\begin{theorem} \label{bip}
Let $S$ be a subtraction set with $\gcd(S) = 1$. The subtraction game $\S$ is bipartite if and only if $1 \in S$ and the elements of $S$ are all odd.
\end{theorem}

\begin{remark} \label{R.proof}
The key results of this paper are the determinations of periodic nim-values and expansion sets in several specific cases. We do not give the proofs for all results. There are considerable similarities in the proofs of various results. We give one proof for the expansion set of the subtraction set $\{a, b\}$ and one proof for the nim-sequence for the subtraction game $\S(1,a,b)$. The reason is simply that these two games are addressed early in the paper. The details of proofs for other stated results are similar; some are quite tedious but they are all straightforward.
\end{remark}

This paper is organized as follows. Section \ref{S.result} lists our results on 5 classes of subtraction sets: $\{a, b\}$, $\{1, a, b\}$, $\{a, b, a+b\}$, $\{a, b, a+b+1\}$, and $\{a, b, a+b+2ja\}$. We consider various cases for each class. Section \ref{S.proof} gives two proofs as mentioned in Remark \ref{R.proof}. We provide proofs in Section \ref{S.proof} rather than in Section \ref{S.result} so that the reader can follow Section \ref{S.result} easily. In Section \ref{S.ubs}, we present a family of ultimately bipartite subtraction games and we examine the expansion sets of subtraction sets of ultimately bipartite subtraction games. The section ends with a conjecture that has become apparent during this work.

We include a code written in Maple software in the appendix. The code receives input $(S, n)$ in which $S$ is the subtraction set and the code will calculate the sequence $\{\G(i)\}_{0 \leq i \leq n}$ to identify the nim-sequence. If the input $n$ is large enough (at least $n_0+p+s_k+1$), the code will return: pre-period and period length, and pre-periodic and periodic nim-values, together with the expansion set $S^{ex}$. When $n$ is not large enough, the code will return nothing, requiring a larger $n$.

Our results also provide evidential support for the claim that the understanding of the pattern of the periodicity of subtraction games and the expansion sets of their subtraction sets is very far from being complete. For example, slight differences on the subtraction sets in Table \ref{T.2.3.3} would result in considerable modifications of both pre-period and period lengths as well as expansion sets.


\section{Results of nim-sequences and expansion sets}   \label{S.result}


We first study the expansion set of the subtraction set $\{a,b\}$ where $a$ and $b$ are relatively prime. When $a=1$,  we only consider the case where $b$ is even since the case where $a=1$ and $b$ is odd is treated by  Theorem \ref{bip}.

Berlekamp et al. \cite[p.~530]{ww3} showed that if $b=ta+r$ such that $t \geq 1$ and $0 \leq r < a$, the subtraction game $\S(a,b)$ is purely periodic with period length $a + b$ and periodic nim-values
\begin{align*}
\begin{cases}
(0^a1^a)^{\frac{t}{2}} 0^r2^{a-r} 1^r,  &\text{if $t$ is even};\\
(0^a1^a)^{\frac{t+1}{2}} 2^r,           &\text{otherwise}.
\end{cases}
\end{align*}
Moreover, they also observed that  the expansion set $S^{ex}$ is periodic in the sense that if $n\in S^{ex}$, then $n+a+b \in S^{ex}$.  We generalize this observation in Theorem \ref{The-sub-2}.

\smallskip
\begin{theorem}  \label{The-sub-2}
Let $a$ and $b$ be relatively prime, positive integers such that $a < b$. Suppose furthermore that if $a=1$, then $b$ is even. Consider the subtraction game $\S(a,b)$. If $a+1 < b \leq 2a$, then the subtraction set has expansion set
\[ \{a,a+1, \ldots, b\}^{*(a+b)}.\]
If $a = 1$, or $b=a+1$, or $b > 2a$, the subtraction set is non-expandable.
\end{theorem}
\begin{proof}
The proof is provided in Section \ref{S.proof}.
\end{proof}


We now examine the periodicity and the expansion sets of some subtraction games whose subtraction sets are $\{1,a,b\}$. For the case where $a$ is odd, Theorem \ref{bip} solves the case where $b$ is odd. When $b$ is even, the game is purely periodic. The periodic nim-values and the expansion set are described in Theorem \ref{1ab}.

\smallskip
\begin{theorem} \label{1ab}
Let $a$ be an odd positive integer and let $b$ be an even positive integer. The subtraction game $\S(1,a,b)$ is purely periodic with the period $a+b$ and the periodic nim-values
\[(01)^{\frac{b}{2}} (23)^{\frac{a-1}{2}}2.\]
Moreover, the subtraction set has expansion set
\[ \{ \{1, 3, \ldots, a \} \cup \{ b,b+2, \ldots, b+a-1\}  \} ^{\ast(a+b)}.\]
\end{theorem}

\begin{proof}
The proof is provided in Section \ref{S.proof}.
\end{proof}

For the case where $a$ is even, we represent $b = ka+r$ where $r < a$. Our results cover the cases $1 \leq k \leq 3$. Table \ref{T.2.3.1} (resp.~Table \ref{T.2.3.2}, Table \ref{T.2.3.3}) lists periodic nim-values and expansion sets of those games corresponding to the case $k = 1$ (resp.~$k =2$, $k = 3$). (The case $k = 1$ and $r=1$ (corresponding to $b = a + 1$) is analyzed in \cite[p.~530]{ww3}.) In our tables, we use ``period" rather than ``period length".

\begin{remark}
To make the tables easily readable, we do not include the pre-period lengths and pre-periodic nim-values of non-purely periodic games. We provide them in \ref{B1} (Tables \ref{T.2.3.1.b}, \ref{T.2.3.2.b}, \ref{T.2.3.3.b}).
\end{remark}

\begin{remark}
We place symbol ``$\cdots$" in front of nim-sequences of those games that are ultimately periodic to distinguish from purely periodic games.
\end{remark}

We now examine subtraction game $\S(a,b,a+b)$ with $2 \leq a$. Note that $b$ can be represented as $ka+r$ for some integers $k$ and $r$ such that $0 \leq r \leq a-1$. Recall from \cite[p.531]{ww3} that  when $k$ is odd, the game $\S(a,b,a+b)$ is purely periodic with period length $(a+b+ka)$ and periodic nim-values
\[ (0^a1^a)^{\frac{k+1}{2}} (2^a3^a)^{\frac{k-1}{2}} 2^a3^r.\]
We give the expansion set of this game in Table \ref{T.2.4}.

When $k$ is even, it is claimed in \cite[p.~531]{ww3} that the game is purely periodic with period length $(2b+r)a$. We have not seen a proof of this claim, and at present we do not know the nim-sequence and expansion set for this case.

We next give some results for subtraction games $\S(a,b,a+b+1)$ in Table \ref{T.2.6}. We obtain results for those cases where $b = 2a+r$ with $r \leq \min(2,a-1)$. Pre-period lengths and pre-periodic nim-values for those games that are non-purely periodic are provided in Table \ref{T.2.6.b} of \ref{B1}.

Finally, we study subtraction games $\S(a,b,b+2ja)$. Note that we can write $b = ka+r$ for some positive $k$ and $r$ such that $r < a$. The case where $r = 0$ is skipped here as it is studied before. We deal with two general cases with $j \in \{1, 2\}$ and one special case with $j = 3$ and $r = 1$ as shown in Table \ref{T.2.5}.



\newgeometry{bottom=2cm}

\begin{landscape}


\begin{center}
\begin{tabular}{|c|c|c|c|c|c|c|}\hline
\begin{tabular}{c}
Subtraction set  \\
$\{1,a,b\}$    \\
$a$: even \\
$b = a+r$
\end{tabular} & $r$     & period $p$
                                  & periodic nim-values
                                            & expansion set \\ \hline
$r = 2$       & \multirow{ 2}{*}{even}
                        & $a+1$   & $(01)^{\frac{a}{2}} 2 $
                                            & non-expandable \\ \cline{1-1} \cline{3-5}

$r > 2$       &         & $a+b$   & $(01)^{\frac{a}{2}} 2 (01)^{\frac{a}{2}} (23)^{\frac{r}{2}-1} 2$
                                            & $\{1,a,a+2,\ldots,b,a+b-1 \}^{\ast p}$ \\ \hline
$r = 1$       & \multirow{4}{*}{odd}
                        & $a+b-1$  & $(01)^{\frac{a}{2}} (23)^{\frac{a}{2}}$
                                            & \begin{tabular}{l}
                                            $\{ \{1,3, \ldots,a-1\} \cup \{ a \}$ \\
                                            $\cup \{ a+1, a+1+2 , \ldots, a+b-2\} \}^{\ast p}$
                                            \end{tabular}  \\ \cline{1-1} \cline{3-5}
\begin{tabular}{c}
$r < a-3$
\end{tabular}
              &         & \multirow{ 2}{*}{$a+b$}
                                    & \multirow{ 2}{*}{
                                    $\cdots 2(32)^{\frac{a-r-3}{2}}(01)^{\frac{r+1}{2}} 2 (01)^{\frac{a-2}{2}} 2 (01)^{\frac{r+1}{2}} $
                                    }
                                            & $\{1,a,b\} \cup \{b+2,a+b+1, 2a+b\}^{\ast p}$  \\  \cline{1-1} \cline{5-5}
\begin{tabular}{c}
$r = a-3$
\end{tabular}
              &         &           &
                                            & $\{1,a,b\} \cup \{b+2\}^{\ast p}$ \\  \cline{1-1} \cline{3-5}
\begin{tabular}{c}
$r = a-1$
\end{tabular}
                &       & $b+1$      & $(01)^{\frac{a}{2}} 2 (01)^{\frac{a-2}{2}} 2$
                                            & non expandable  \\ \hline
\end{tabular}
\begin{table}[ht]
\caption{Subtraction games $\S(1,a,b)$ with $b = a+r$.} \label{T.2.3.1}
\end{table}
\end{center}


\begin{center}
\begin{tabular}{|c|c|c|c|c|c|c|}\hline
\begin{tabular}{c}
Subtraction set \\
\{1,a,b\}  \\
$a$: even \\
$b = 2a+r$
\end{tabular}           & $r$   & period $p$    & periodic nim-values    & expansion set \\ \hline
$r = 0$, $a \geq 4$     &  \multirow{5}{*}{even}
                                & \multirow{4}{*}{$a+b$}
                                                & $(01)^{\frac{a}{2}} 2 (01)^{\frac{a}{2}} (23)^{\frac{a-2}{2}} 2$
                                                & $\{ \{ 1 \} \cup \{ a,a+2, \ldots,b \} \cup \{a+b-1\} \} ^{\ast p}$  \\ \cline{1-1} \cline{4-5}
$2 = r < a-4$        &       & 
                                                & $\cdots 2(32)^{\frac{a}{2}-3}(01)^2 2 ((01)^{\frac{a}{2}-1})^2 2 (01)^2$
                                                                        &  non-expandable   \\ \cline{1-1} \cline{4-5}
$4 \leq r < a-4$

                        &       & 
                                                & $\cdots 2(32)^{\frac{a}{2}-4}(01)^3 2 ((01)^{\frac{a}{2}} 2)^2 (01)^3$
                                                                        & \multirow{2}{*}{$\{1,a,b,b+2, a+b+1, a+2b+2\}$}   \\ \cline{1-1} \cline{3-4}
$4 < r = a-4$

                        &       & $a-1$
                                                & \begin{tabular}{c}
                                                 $\cdots 2 (01)^{\frac{a}{2}-1}$
                                                 \end{tabular}
                                                                        &     \\ \cline{1-1} \cline{3-5}

$4 \leq r = a-2$         &       & $b+1$        & $\cdots 2 (01)^{\frac{a-2}{2}} 2 (01)^{\frac{a}{2}} 2 (01)^{\frac{a-2}{2}}$
                                                                        & $\{1,a,b\} \cup \{b+2, a+b+1\}   ^{\ast p}$ \\    \hline
$r = 1, 3$              &  \multirow{ 2}{*}{odd}
                                & $a+1$       & $(01)^{\frac{a}{2}} 2$
                                                                        &  non-expandable (same as $S(1,a)$)  \\ \cline{1-1} \cline{3-5}
$5 \leq r \leq  a-1$    &       & $a+b$      & $((01)^{\frac{a}{2}} 2)^2 (01)^{\frac{a}{2}} (23)^{\frac{r-3}{2}} 2$
                                                                        & $\{1,a,a+2,b-2,b,a+b-1\}^{\ast p}$  \\ \hline
\end{tabular}
\begin{table}[ht]
\caption{Subtraction games $\S(1,a,b)$ with $a$ being even and $b = 2a+r$.} \label{T.2.3.2}
\end{table}
\end{center}


\begin{center}
\begin{tabular}{|c|c|c|c|c|c|c|c|}\hline
\begin{tabular}{c}
Subtraction set \\
\{1,a,b\} \\
$a$: even \\
$b = 3a+r$
\end{tabular}
            & $r$     & period $p$
                               & periodic nim-values
                                            & expansion set \\ \hline
$r = 0$, $a \geq 4$
            & \multirow{4}{*}{even}
                      & $b+1$  & $( (01)^{\frac{a}{2}} 2 )^2 (01)^{\frac{a}{2}-1} 2$
                                            & $\{ 1,a,2a+1,b \} ^{\ast p}$ \\ \cline{1-1} \cline{3-5}
$r = 2, 4$  &         & $a+1$  & $(01)^{\frac{a}{2}} 2$
                                            & non-expandable \\ \cline{1-1} \cline{3-5}
$r = 6$     &         & \multirow{6}{*}{$a+b$}
                               & \multirow{2}{*}{
                                 $((01)^{\frac{a}{2}} 2 )^4 (32)^{\frac{r}{2}-2}$
                                 }
                                            & $\{ 1, a, a+2, 2a+3, b-2, b, a+b-1 \}^{\ast p}$ \\ \cline{1-1} \cline{5-5}
$r > 6$     &         & 
                                & 
                                            & \multirow{2}{*}{
                                               $\{ 1,a,a+2,b-2, b, a+b-1 \}^{\ast p}$
                                               }
                                                             \\ \cline{1-2} \cline{4-4}
$r = 1$     & \multirow{4}{*}{odd}
                      & 
                                & $( (01)^{\frac{a}{2}} 2 )^3 (32)^{\frac{a}{2}-1}$
                                                             & 
                                                             \\ \cline{1-1} \cline{4-5}
$3 = r < a-5$&        & 
                                & \multirow{3}{*}{
                                $\cdots 2 (32)^{\frac{a-r-5}{2}} (01)^{\frac{r+3}{2}} 2 ((01)^{\frac{a}{2}-1}2)^3 (01)^{\frac{r+3}{2}}$
                                }
                                            & non-expandable
                                            \\ \cline{1-1} \cline{5-5}
$3 <r < a-5$&         & 
                                &
                                            & \multirow{3}{*}{
                                             $\{ 1, a,b, b+2, a+b+1, a+2b+2 \}$
                                             }
                                             \\ \cline{1-1} \cline{3-3}
$5 \leq r = a-5$&     & $a-1$
                                &
                                            & 
                                            \\ \cline{1-1} \cline{3-4}
$5 \leq r = a - 3$ &         & \multirow{2}{*}{
                                $b+1$
                                }       & \multirow{2}{*}{
                                          $\cdots 2 (01)^{\frac{a}{2}-1} (2 (01)^{\frac{a}{2}})^2 2 (01)^{\frac{a}{2}-1}$
                                          }
                                            & 
                                            \\ \cline{1-1} \cline{5-5}
$5 \leq r = a-1$   &         &          &
                                            & $\{1,a,b\} \cup \{ b+2, a+b+1\}^{\ast p}$  \\ \hline
\end{tabular}
\begin{table}[ht]
\caption{Subtraction games $\S(1,a,b)$ with $a$ being even and $b = 3a+r$.} \label{T.2.3.3}
\end{table}


\begin{tabular}{|c|c|c|c|c|}\hline
Subtraction set    & $k$    & period $p$     & periodic nim-values   & expansion set \\ \hline
\begin{tabular}{c}
\{a,b,a+b\} \\
$a \geq 2$ \\
$b = ka+r$ \\
$0 < r < a$
\end{tabular}
                   & odd     & $a+2b-r$   & $(0^a1^a)^{\frac{k+1}{2}}(2^a3^a)^{\frac{k-1}{2}}2^a3^r$
                                                  & \begin{tabular}{c}
                                                    $\big{\{} \{sa: s \text{ is odd}, s \leq k\}$ \\
                                                        $\cup \{ ka+1, ka+2, \ldots, b+a \}$ \\
                                                        $\cup \{b+sa: s \text{ is odd}, s \leq k \} \big{\}}^{\ast p}$
                                                    \end{tabular} \\  \cline{2-5}
                   & even    & $(2b+r)a$   & unknown   & unknown \\ \hline
\end{tabular}
\begin{table}[ht]
\caption{Subtraction games $\S(a,b,a+b)$.} \label{T.2.4}
\end{table}
\end{center}
\end{landscape}


\begin{landscape}
\begin{center}
\begin{tabular}{|c|c|c|c|c|c|c|}\hline
\begin{tabular}{c}
Subtraction set \\
$\{a,b,b+a+1\}$ \\
$a \geq 4$\\
$b = 2a+r$ \\
$r \leq 2$
\end{tabular}
     & period $p$  & periodic nim-values
                                                                 & expansion set \\ \hline
$\{a,2a,3a+1\}$
                    & \multirow{3}{*}{
                    $b+2a+1$
                    }        & $\cdots 1 0^{a-1} 2 1^{a-1} 0 2^{a-1} 1 0 3^{a-2} 2$
                                                                & $\{a\} \cup \{2a,3a+1, 5a \}^{\ast p}$
                                                                \\ \cline{1-1} \cline{3-4}
$\{a,2a+1,3a+2\}$   &      & $0^a 1^a 0 2^{a-1} 10 3^{a-2} 21$
                                                & non-expandable  \\ \cline{1-1} \cline{3-4}
\{a,2a+2,3a+3\}
                    &
                                        & $\cdots 2^{a-2} 11 00 3^{a-4} 22 11 0^{a-1} 2 1^{a-1} 00$
                                                &   \begin{tabular}{c}
                                                    $\{a,2a+2,3a+3\} \cup \{3a+4\}^{\ast p}$ \\
                                                    and non-expandable if $a = 4$\\
                                                    \end{tabular}   \\ \hline
\end{tabular}
\begin{table}[ht]
\caption{Subsection games $\S(a,b,b+a+1)$ with $a \geq 4$, $b = 2a+r$, and $0 \leq r \leq 2$.} \label{T.2.6}
\end{table}


\begin{tabular}{|c|c|c|c|c|}\hline
\begin{tabular}{c}
Subtraction set \\
$\{a,b,b+2ja\}$
\end{tabular}      & period $p$ & $k$     & periodic nim-values  & expansion set \\ \hline
\multirow{2}{*}{
\begin{tabular}{c}
\{a,b,b+2a\}
\end{tabular}
}
                    & \multirow{2}{*}{
                    $a+b$
                    }  & even   & $(0^a1^a)^{\frac{k}{2}} 0^r2^{a-r} 1^r$
                                                & \multirow{2}{*}{
                                                non-expandable
                                                }   \\ \cline{3-4}

                    &  & odd  & $(0^a1^a)^{\frac{k+1}{2}} 2^r$
                                                &      \\ \hline
\multirow{2}{*}{
\begin{tabular}{c}
\{a,b,b+4a\}\\
$k \geq 3$
\end{tabular}
}
                    & \multirow{2}{*}{
                    $2b+4a$
                    }
                       & even  & $((0^a1^a)^{\frac{k}{2}} 0^r2^{a-r} 1^r )^2 3^{a-r} 0^r 2^{a-r} 1^r$

                                                &  \multirow{2}{*}{
                                                 \begin{tabular}{c}
                                                    $\{a ,b, b+2a, b+4a,$ \\
                                                    $  2b+3a \}^{\ast p}$
                                                    \end{tabular}
                                                 } \\  \cline{3-4}
          &        & odd   & $((0^a1^a)^{\frac{k+1}{2}} 2^r)^2 0^{a-r} 3^r 1^{a-r} 2^r$ &
                                                    \\ \hline

\begin{tabular}{c}
\{a,b,b+6a\}\\
$k \geq 3$, $r = 1$
\end{tabular}
                    & $2b+6a$  & odd  & $((0^a1^a)^{\frac{k+1}{2}} 2)^2 (0^{a-1}3 1^{a-1}2)^2$
                                                &   \begin{tabular}{c}
                                                   $\{ a, b, b+2a, b+4a,$ \\
                                                   $b+6a, 2b+5a \}^{\ast p}$
                                                    \end{tabular}   \\ \hline
\end{tabular}
\begin{table}[ht]
\caption{Subtraction games $\S(a,b,b+2ja)$ with $a \geq 2$, $b = ka+r$.} \label{T.2.5}
\end{table}
\end{center}
\end{landscape}

\begin{landscape}
\begin{center}
\begin{tabular}{|c|c|c|c|c|l|}\hline
\multicolumn{2}{|c|}{$k$}                              & $i$       & $t$  & $n$       & $\G(n+s) = \G(n)$       \\ \hline
\multirow{2}{*}{even}    & \multicolumn{1}{|c|}{$k<t$} &           &      & $a+b$     & $\G(a+b+ka+i) = \G(ka+i) = 0 = \G(0) = \G(a+b)$  \\ \cline{2-6}
      & \multicolumn{1}{|c|}{$k=t$} & $i<r$     &      & $(t-1)a$  & \begin{tabular}{c}
                                                                  $\G((t-1)a+ta+i) = \G((t-2)a-(r+i)+ta+a+r)$ \\
                                                                  $= \G((t-2)a-(r+i)+a+b) = \G((t-2)a-(r+i)) = 1 = \G((t-1)a)$
                                                                  \end{tabular}   \\ \hline
\multirow{8}{*}{odd}    & $k=1$ & $i>0$     &      & $a-1$     & $\G(a-1+a+i) = \G(2a+i-1) = 0 = \G(a-1)$ \\ \cline{2-6}
    & \multirow{7}{*}{$k>1$}    & $i\leq r$ & \multirow{3}{*}{even}      & $b-a-i$   & $\G(b-a-i+ka+i) = \G(a+b+(k-2)a) = \G((k-2)a) = 1 = \G(b-a-i)$ \\ \cline{3-3} \cline{5-6}
    &       & $i>r=0$   &      & $b-a$     & $\G(b-a+ka+i) = \G(a+b+(k-2)a+i) = \G((k-2)a+i) = 1 
    = \G(b-a)$  \\ \cline{3-3} \cline{5-6}
    &       & $i>r>0$   &      & $a-i+r-1$     & $\G(a-i+r-1+ka+i) = \G\big((k+1)a+r-1\big) = 0
= \G(a-i+r-1)$ \\ \cline{3-6}
    &       & $i<r$     & \multirow{4}{*}{odd}      & $ta$      & $\G(ta+ka+i) 
    = \G\big((k-1)a-(r-i)\big) = 1
    = \G(ta)$  \\ \cline{3-3} \cline{5-6}
    &       & $i=r = 0$ &      &           & not hold as we require $\gcd(a,b) = 1$ \\  \cline{3-3} \cline{5-6}
    &    & $i=r > 0$    &   & $a-1$     & $\G(a-1+ka+i)= \G\big((k+1)a+i-1\big) = 0 = \G(a-1)$  \\ \cline{3-3} \cline{5-6}
    &       & $i>r$     &      & $(t+1)a-1$& $\G((t+1)a-1+ka+i)= \G(ka+i-r-1) =1 = \G((t+1)a-1)$  \\ \hline
\end{tabular}
\begin{table}[ht]
\caption{$b = ta+r > 2a; a+1 \leq s = ka+i \leq b - 1$.} \label{T.s}
\end{table}
\end{center}
\end{landscape}

\restoregeometry

\bigskip
\section{Two proofs} \label{S.proof}

As outlined in Remark \ref{R.proof}, we give in this section two proofs, one for the expansion set of the subtraction set $\{a,b\}$ claimed in Theorem \ref{The-sub-2} and another one for the nim-sequence for the subtraction game $\S(1, a, b)$ claimed in Theorem \ref{1ab}.

\subsection{Proof of Theorem \ref{The-sub-2}} \label{P.2.1}

Recall that the game $\S(a,b)$ has periodic nim-values
\begin{align} \label{l0}
\begin{cases}
(0^a1^a)^{\frac{t}{2}} 0^r2^{a-r} 1^r,  &\text{if $t$ is even};\\
(0^a1^a)^{\frac{t+1}{2}} 2^r,           &\text{otherwise}. 
\end{cases}
\end{align}
in which $b = ta+r$ \cite[p. 530]{ww1}. We now prove the expansion set. The proof requires checking several cases, using Theorem \ref{ext.form}. The outline of the proof is as follows.
\begin{enumerate} [(a)] \itemsep0em
\item For each of the three cases $a = 1$, $b = a+1$, and $b > 2a$, we show that for every positive integer $s$, there exists a nonnegative integer $n$ such that $\G(n+s) = \G(n)$ and so the subtraction set $\{a, b\}$ is non-expandable.
\item For $a+1 < b \leq 2a$, we show that $\G(n+s) \neq \G(n)$ for all nonnegative integers $n$ if and only if $s \in \{a,a+1, \ldots, b\}^{\ast (a+b)} = X$ and so the subtraction set $\{a, b\}$ is expandable to $X$.
\end{enumerate}
By Theorem \ref{ext.form}, in the two cases \textrm{(a)} and \textrm{(b)} we only need to examine $s < a+b$ with $s \notin \{a, b\}$.

\begin{enumerate} [(a)] \itemsep0em
\item We first consider the case where $a = 1$ or $b = a+1$ or $b > 2a$.
    \begin{enumerate} [(i)] \itemsep0em
    \item Consider the case $a = 1$. We only examine the case $b$ is even. Let $s \leq a+b-1$ such that $s \notin\{1, b\}$. Recall from (\ref{l0}) that the first $b$ nim-values $\G(n)$ (for $n$ from $0$ to $b-1$) are $(01)^{\frac{b}{2}}$. In particular,
        \begin{align} \label{l1}
        \G(n) = 1 \quad  \text{if $n \leq b-1$ and $n$ is odd}.  
        \end{align}
        \begin{enumerate} [$\bullet$] \itemsep0em
        \item If $s$ is even then $s \leq b-2$ as $s < b$ and $b$ is even, implying that $s+1$ is odd and $s+1 \leq b-1$. Let $n = 1$. By (\ref{l1}), we have $\G(1+s) = 1 = \G(1)$.
        \item If $s$ is odd, let $n = b-1$. As $s-2$ is odd and $1 \leq s-2 \leq b-3$, $\G(s-2) = 1 = \G(b-1)$  by (\ref{l1}). Also note that $\G(b+1+s-2) = \G(s-2)$ since $p = b+1$. Therefore,
        $\G(b-1+s) = \G(b+1+s-2) = \G(s-2) = \G(b-1).$
        \end{enumerate}
    \item Consider the case $b = a+1$. We examine two intervals for $s \leq a+b-1$:
        \begin{align*}
        \begin{cases}
        1 \leq s \leq a-1, \\
        a+2 \leq s \leq a+b-1.
        \end{cases}
        \end{align*}
        Recall that the first $4a+2$ nim-values $\G(n)$ (for $n$ from $0$ to $4a+1$) is $0^a1^a20^a1^a2$. In other words, for $n \leq 4a+1$, we have
        \begin{align}  \label{l2}
        \G(n) =
        \begin{cases}
        0, \text{if either $0 \leq n \leq a-1$ or $2a+1 \leq n \leq 3a$}  \\ 
        1, \text{if either $a \leq n \leq 2a-1$ or $3a+1 \leq n \leq 4a$}   \\
        \end{cases}
        \end{align}
        \begin{enumerate} [$\bullet$] \itemsep0em
        \item For $1 \leq s \leq a-1$, choose $n = a$. Then $a+1 \leq a+s \leq 2a-1$. By (\ref{l2}), $\G(a+s) = 1 = \G(a)$.
        \item For $a+2 \leq s \leq a+b-1$, choose $n = a-1$. Then
        $0 \leq a-1+s-(a+b) \leq a-2$ and so
        $\G(a-1+s-(a+b)) = 0 = \G(a-1)$ by (\ref{l2}). By the periodicity, $\G(a-1+s) = \G(a-1+s-(a+b))$. Therefore, $\G(a-1+s) = \G(a-1)$.
        \end{enumerate}
        \item Consider the case $b > 2a$. We examine three intervals for $s \leq a+b-1$:
        \begin{align*}
        \begin{cases}
        1 \leq s \leq a-1 \\
        a+1 \leq s \leq b-1 \\
        b+1 \leq s \leq a+b-1.
        \end{cases}
        \end{align*}
        \begin{enumerate} [$\bullet$] \itemsep0em
        \item Consider the case $1 \leq s \leq a-1$. Let $n = a$. Recall that $\G(m) = 1$ for $a \leq m \leq 2a-1$ by (\ref{l0}) and so $\G(a+s) = 1 = \G(a)$ as $a+1 \leq a+s \leq 2a-1$.
        \item Consider the case $a+1 \leq s \leq b-1$. We separate this case into many subcases in Table \ref{T.s}. It can be checked that for each $s$ such that $a+1 \leq s \leq b-1$, there exists $n$, as shown in Table \ref{T.s}, such that $\G(n+a) = \G(n)$. We leave the calculation to the reader.
        \item Consider the case $b+1 \leq s \leq a+b-1$. We have $0 \leq s-b-1 \leq a-2$. Let $n = a-1$. Recall that $\G(m) = 0$ for $0 \leq m \leq a-2$ by (\ref{l0}), and so $\G(a-1+s-(a+b)) = \G(s-b-1) = 0 = \G(a-1)$. By the periodicity, $\G(a-1+s) = \G(a-1+s-(a+b))$. Therefore, $\G(a-1+s) = \G(a-1)$.
        \end{enumerate}
    \end{enumerate}
\item We now consider the case  $a+1 < b \leq 2a$. We examine three intervals for $s$:
        \begin{align*}
        \begin{cases}
        1 \leq s \leq a-1, \\
        a < s < b, \\
        b+1 \leq s \leq a+b-1.
        \end{cases}
        \end{align*}
     The case $1 \leq s \leq a-1$ and the case $b+1 \leq s \leq a+b-1$ can be treated similarly to the case \textrm{(a).(iii)}.

     Consider the case $a < s < b$. We show that $\G(n+s) \neq \G(n)$ for all $n$. Recall from (\ref{l0}) that the game $\S(a,b)$ with $a+1 < b \leq 2a$ has the periodic nim-values
    \begin{align*}
    \begin{cases}
    0^a 1^a 2^a,     &\text{if $b = 2a$};\\
    0^a 1^a 2^ r,    &\text{if $b=a+r$}.
    \end{cases}
    \end{align*}
    It can be checked that if $\G(n)$ is in a block $0's$, then $\G(n+s)$ is either in a block $1's$ or in a block $2's$; if $\G(n)$ is in a block $1's$, then $\G(n+s)$ is either in a block $2's$ or in a block $0's$; if $\G(n)$ is in a block $2's$, then $\G(n+s)$ is either in a block $0's$ or in a  block $1's$. Therefore, $\G(n+s) \neq \G(n)$ for all $s$ and $n$ such that $a < s < b$ and $n \geq 0$.
\end{enumerate}


\subsection{Proof of Theorem \ref{1ab}} \label{P.2.2}

We will give the proof for the nim-sequence only. The expansion set can be verified by using a similar technique in Theorem \ref{The-sub-2}.

We will show that the first $a+2b$ nim-values $\G(n)$, for $0 \leq n \leq a+2b-1$, are
\[ (01)^{\frac{b}{2}} (23)^{\frac{a-1}{2}}2 (01)^{\frac{b}{2}}\]
and so the condition $\G(n+a+b) = \G(n)$ holds for $0 \leq n \leq b-1$. Lemma \ref{per.sub} then implies that the subtraction game $\S(1,a,b)$ is purely periodic with the period length $a+b$ and periodic nim-values
\[(01)^{\frac{b}{2}} (23)^{\frac{a-1}{2}}2.\]

We divide the first $a+2b$ nim-values $\G(n)$, for $n$ from $0$ to $a+2b-1$, into three blocks: 
\begin{enumerate} \itemsep0em
\item $\G(0)\ldots \G(b-1)$, 
\item $\G(b)\ldots \G(a+b-1)$, 
\item $\G(a+b)\ldots \G(a+2b-1)$. 
\end{enumerate}

It can be checked by induction on $n$ that the block $(1)$ forms the sequence $(01)^{\frac{b}{2}}$, the block $(2)$ forms the sequence $(23)^{\frac{a-1}{2}}2$, and the block $(3)$ forms the sequence $(01)^{\frac{b}{2}}$. Thus the first $a+2b$ nim-values are
$(01)^{\frac{b}{2}} (23)^{\frac{a-1}{2}}2 (01)^{\frac{b}{2}}$, as required.


\section{More on ultimately bipartite subtraction games} \label{S.ubs}

Some ultimately bipartite games were exhibited in \cite{bi}. In particular the games with subtraction sets $\{3,5,2^k+1\}$, for  $k\geq 3$, and  $\{a,a+2,2a+3\}$, for odd $a\geq 3$, are all ultimately bipartite.  In this section, we first introduce another family of ultimately bipartite games with 3-element subtraction sets. We then consider the expansion sets of  subtraction sets of ultimately bipartite games, with a conjecture being given at the end.


\begin{theorem} \label{U.B}
For odd integer $a \geq 5$, the subtraction game $\S(a,2a+1,3a)$ is ultimately bipartite with pre-period length $2a^2-a-1$. The pre-periodic nim-values are shown in Table \ref{table-ub}.
\end{theorem}

\begin{proof} Let $n_0 = 2a^2-a-1$. We show that for $n \geq n_0$, $\G(n) = 0$ if $n$ is even and $\G(n) = 1$ otherwise.
\begin{enumerate}[(1)] \itemsep0em
\item We place the sequence of first $n_0$ nim-values $\G(0)$, $\G(1)$, \ldots, $\G(n_0-1)$ in Table \ref{table-ub} from left to right and top to bottom with empty cell being empty and the cell with the notation $\underset{b}{\underbrace{}}$ containing $b$ consecutive integers and so on. The nim-values in Table \ref{table-ub} can be checked by induction on $n$ for $n < n_0$.

    (Each row from the second to the second last of Table \ref{table-ub} has length $4a+2$. Observe that on each of these rows, there are exactly $a$ numbers in each two consecutive blocks A and B, B and C, C and D, F and G, G and H.)

\item We next prove that, after the first $n_0$ nim-values in Table \ref{table-ub}, the game has alternate nim-values $0,1,0,1,\ldots$ (starting from $\G(n_0)$.) The proof consists of six steps.
    \begin{enumerate} [(a)] \itemsep0em
    \item Let $A_1$ be the set of all nonnegative integers $n < n_0$ such that $\G(n) = 0$ in Table \ref{table-ub}. Let $A_2$ be the set of all even integers $n  \geq n_0$. Set $\A = A_1 \cup A_2$.
    \item We then prove the two facts:
        \begin{enumerate} \itemsep0em
        \item for every position $n$ in $\A$, all the moves from $n$ do not terminate in $\A$,
        \item for every position $n$ not in $\A$, there is a move from $n$ that terminates in $\A$.
        \end{enumerate}
    \item It follows from step \textrm{(b)} and the definition of nim-values that $\A$ is the set of positions whose nim-values are zero, by which we have
        \begin{enumerate} \itemsep0em
        \item [\textrm{(iii).}] $\G(n) = 0$ for all even $n$ such that $n \geq n_0$.
        \end{enumerate}
    \item For all odd $n \geq n_0+3a$, we have $n-a$, $n-(2a+1)$, and $n-3a$ are all in $\A$. By \textrm{(iii)} and the definition of nim-values, we have
        \begin{enumerate} \itemsep0em
        \item [\textrm{(iv).}] $\G(n) = 1$ for all odd $n$ such that  $n \geq n_0+3a$.
        \end{enumerate}
    \item We now examine $\G(n)$ for odd $n$ such that $n_0 \leq n \leq n_0+3a-1$. The $3a$ nim-values $\G(n)$ for $n_0 \leq n \leq n_0+3a-1$ continuing those in Table \ref{table-ub} is $0101\ldots01$ (the last $3a$ values in Table \ref{table-ub1}, being located in the bottoms of blocks F, G, H, I, and B), by which we have
        \begin{enumerate} \itemsep0em
        \item [\textrm{(v).}] $\G(n) = 1$ for all odd $n$ such that $n_0 \leq n \leq n_0+3a-1$.
        \end{enumerate}
    \item It follows from \textrm{(iii)}, \textrm{(iv)}, and \textrm{(v)} that the game ultimately has alternate nim-values $0,1,0,1,\ldots$ with pre-period length $n_0$.
    \end{enumerate}
     Step \textrm{(a)} is the definition of $\A$. Steps \textrm{(c)}, \textrm{(d)}, and \textrm{(f)} are straightforward. Steps \textrm{(b)} and \textrm{(e)} can be verified by induction on $n$. Checking several cases is required but is also straightforward and we leave the calculation to the reader.
\end{enumerate}
\end{proof}

\begin{remark}
The relevance of Theorem \ref{U.B}, beyond its general interest, is that it indicates that the ultimately bipartite case, which is the simplest form of ultimate periodicity, can nevertheless entail a high level of complexity at the beginning of the sequence.
\end{remark}


\begin{landscape}
\begin{center}
\begin{tabular}{||c|c||c|c||c||c|c||c|c||}\hline
{\small Block A}                             & {\small Block B}                        & {\small Block C}                          &{\small Block D}                            &{\small Block E}
                &{\small Block F}                            &{\small Block G}                          &{\small Block H}                              &{\small Block I} \\ \hline
                                    &                                &                                  &                                   &
                & $\underset{a}{\underbrace{000\ldots000}}$     &                     & $\underset{a}{\underbrace{111\ldots111}}$         & $0$    \\ \hline
$\underset{a-1}{\underbrace{22\ldots22}}$     & $1$  & $\underset{a-1}{\underbrace{33\ldots33}}$  & $0$     & $2$
                & $\underset{a-2}{\underbrace{00\ldots00}}$   & $10$   & $\underset{a-2}{\underbrace{11\ldots11}}$     & $010$    \\ \hline
$\underset{a-3}{\underbrace{2\ldots2}}$     & $101$  & $\underset{a-3}{\underbrace{3\ldots3}}$  & $010$ & $2$
                & $\underset{a-4}{\underbrace{0\ldots0}}$   & $1010$   & $\underset{a-4}{\underbrace{1\ldots1}}$     & $01010$    \\ \hline
$\vdots$                            & $\vdots$                       & $\vdots$                         & $\vdots$                          & $\vdots$
                & $\vdots$                          & $\vdots$                        & $\vdots$ & $\vdots$    \\ \hline
$2222$       & $\underset{a-4}{\underbrace{101\ldots01}}$ & $3333$         & $\underset{a-4}{\underbrace{010\ldots10}}$   & $2$
                & $000$     & $\underset{a-3}{\underbrace{10\ldots10}}$ & $111$       & $\underset{a-2}{\underbrace{010\ldots10}}$    \\ \hline
$22$       & $\underset{a-2}{\underbrace{1010\ldots101}}$ & $33$   & $\underset{a-2}{\underbrace{0101\ldots010}}$ & $2$
                &      &  &       &     \\ \hline
\end{tabular}
\begin{table}[ht]
\caption{The arrangement of the first $(2a+1)(a-1)+3a$ nim-values. $\G(n) = 0$ for $n \in \A$.} \label{table-ub}
\end{table}

\begin{tabular}{||c|c||c|c||c||c|c||c|c||}\hline
{\small Block A}                             & {\small Block B}                        & {\small Block C}                          &{\small Block D}                            &{\small Block E}
                &{\small Block F}                            &{\small Block G}                          &{\small Block H}                              &{\small Block I} \\ \hline
$\vdots$                            & $\vdots$                       & $\vdots$                         & $\vdots$                          & $\vdots$
                & $\vdots$                          & $\vdots$                        & $\vdots$ & $\vdots$    \\ \hline
$2222$       & $\underset{a-4}{\underbrace{101\ldots01}}$ & $3333$         & $\underset{a-4}{\underbrace{010\ldots10}}$   & $2$
                & $000$     & $\underset{a-3}{\underbrace{10\ldots10}}$ & $111$       & $\underset{a-2}{\underbrace{010\ldots10}}$    \\ \hline
$22$       & $\underset{a-2}{\underbrace{1010\ldots101}}$ & $33$   & $\underset{a-2}{\underbrace{0101\ldots010}}$ & $2$
                & {\bf$0$}     & $\underset{a-1}{\underbrace{101\ldots010}}$ & $1$       & $\underset{a}{\underbrace{0101\ldots010}}$    \\ \hline
                                      & $\underset{a}{\underbrace{10101\ldots0101}}$ &
                                &     &                 &      &  & &    \\ \hline
\end{tabular}
\begin{table}[ht]
\caption{The bottom part of Table \ref{table-ub} with extra $3a$ nim-values at the end.} \label{table-ub1}
\end{table}
\end{center}
\end{landscape}

We now give three partial results that place limitations on the  expansion sets of subtraction sets of ultimately bipartite games. Note that $S = \{s_1, s_2, \ldots, s_k\}$. First we recall a well known result.

\begin{lemma} [Ferguson's property] \label{fer} \cite[p. 86]{ww1}
$\G(n) = 0$ if and only if $\G(n+s_1) = 1$ for all $n$.
\end{lemma}


\begin{lemma} \label{n_0start}
Let $\S$ be an ultimately bipartite game. Then $n_0 \geq s_1+s_k$.
\end{lemma}

\begin{proof}
Note that $s_1 > 1$ by Theorem \ref{bip}. Consider the first $s_1+s_k$ nim-values $\G(0), \G(1), \ldots, \G(s_1+s_k-1)$. Note that the first $s_1$ nim-values of the subsequence are all zero and so the last $s_1$ nim-values of the subsequence are non-zero as $\G(m+s_k) \neq \G(m)$ for all $m \leq s_1-1$. Therefore, the periodicity must start at some $n_0 \geq s_1+s_k$.
\end{proof}


\begin{lemma} \label{n_0-1}
Let $\S$ be an ultimately bipartite game. Then $\G(n_0) = 0$, $\G(n_0-1) \geq 2$, $\G(n_0-2) = 0$, and  $\G(n_0-1-s_k) = 1$.
\end{lemma}

\begin{proof}
For the moment let $m=n_0$ if $ \G(n_0) = 0$ and set $m=n_0+1$ if $ \G(n_0) =1$, so that $\G(m) = 0$.  We have
\begin{align} \label{n_0}
\G(m + i) =
\begin{cases}
0, &\text{if $i$ is even};\\
1, &\text{otherwise}.
\end{cases}
\end{align}
Therefore, $\G(m-2+s_1) = 1$ and hence by Ferguson's property, $\G(m-2) = 0$. Notice that $\G(m-1) \neq 1$, as otherwise  the periodicity would commence at $\G(m-2)$ or earlier. In particular, $\G(n_0) = 0$ and so $m=n_0$. Also by Ferguson's property, $\G(n_0-1) \neq 0$, as otherwise $\G(n_0-1+s_1) = 1$ contradicting (\ref{n_0}). So $\G(n_0-1) \geq 2$. We now show that $\G(n_0-1-s_k) = 1$. Note that
\[\G(n_0-1) = mex\{\G(n_0-1-s_i): 1\leq i\leq k\}.\]
For each $i < k$, we have $\G(n_0-1-s_i) \neq \G(n_0-1-s_i+s_k)$. Moreover, by (\ref{n_0}), we have $\G(n_0-1-s_i+s_k) = 1$. Therefore, $\G(n_0-1-s_i) \neq 1$ for all $i < k$. It follows that $\G(n_0-1-s_k) = 1$.
\end{proof}


\begin{lemma} \label{limit-s}
Let $\S$ be an ultimately bipartite game. If $s$ can be adjoined to $S$, then $s < s_k$.
\end{lemma}

\begin{proof}
Assume by contradiction that $s > s_k$ and $s$ can be adjoined to $S$. Note that all $s_i$ are odd (as the nim-sequence is ultimately $01\ldots01$) and greater than one. We also have $s > 1$ and $s$ is odd. By Lemma \ref{n_0start}, we have $n_0 \geq s_1+s$. By Lemma \ref{n_0-1}, $\G(n_0-1-s_k) = 1$ and so $\G(n_0-1-s_k+s) \neq 1$. This contradicts (\ref{n_0}) as $s-1-s_k$ is odd.
\end{proof}


\begin{conjecture} \label{Con1}
The subtraction set of an ultimately bipartite game is non-expandable.
\end{conjecture}

\bigskip

\section*{Acknowledgements}
I am grateful to Dr. Grant Cairns, for several suggestions regarding content and exposition.

 \bibliographystyle{plain}

\appendix

\section{A code, written on Maple, for subtraction game} \label{A}

We include in \ref{A1} the code ``sub" whose input is $(S, n)$ where $S$ is the subtraction set and $n$ is the largest integer the code needs to calculate $\G(n)$ to identify the pre-period length $n_0$ and period length $p$. If $n > n_0 + p + s_k$ the periodicity will be found and the code outputs pre-period length $n_0$, period length $p$, pre-periodic nim-values $\G(0)\ldots\G(n_0-1)$, periodic nim-values $\G(n_0)\ldots\G(n_0+p-1)$, and the expansion set $S^{ex}$. If the code returns nothing, it requires a larger $n$.

The code ``sub" recalls the subcode ``mex" in \ref{A2} and the subcode ``F" in \ref{A3}. In these codes, for arbitrary sets $A, B$ and element $a$, \underline{\textrm{`union`(A, B)}} means $A \cup B$, \underline{\textrm{not `in`(a, A)}} means $a \notin A$, and \underline{\textrm{`minus`(A, B)}} means $A \setminus B$.

We provide two examples in \ref{A4}. The first example is purely periodic and so it does not have pre-periodic nim-values.

\subsection{The code ``sub"} \label{A1}
\begin{verbatim}
sub := proc (S, n)
local i, j, k, l, m, s, g, p, T, Y, Z;
############
for i from 0 to min(S) - 1 do
     g[i] := 0;
end do;
############
for i from min(S) to n do
     T := {};
     for k in S do
        if k <= i then
            T := `union`(T, {g[i-k]});
        end if;
     end do;
     g[i] := mex(T);
end do;
############
for i from 0 to n - max(S) - 1 do
     #########
     for j from i+1 to n - max(S) do
        for l from 0 to max(S) do
            if g[i+l] <> g[j+l] then
                l := 0;
                break;
            end if;
        end do;
        if l = max(S) + 1 then
             break;
        end if;
     end do;
     #########
     if j < n - max(S) then
          p := j - i;
          ######
          print(pre_period length)
          print(i);
          ######
          print(period length);
          print(p);
          ######
          print(pre_periodic nim_values);
          if 0 < i then
               print(seq(g[x], x = 0 .. i - 1));
          end if;
          ######
          print(periodic nim_values);
          print(seq(g[x], x = i .. i + p - 1));
          ######
          if i = 0 then
               Y := S;
               for s from 1 to p-1 do
                    if not `in`(s, S) then
                        for m from 0 to i - 1 do
                            if g[m+s] = g[m] then
                                 break;
                            end if;
                        end do;
                        if m = p then
                             Y: = `union`(Y, {s});
                        end if;
                    end if;
               end do;
               ###
               print(expansion set);
               return Y^(p);
               ###
          else
               Y := S intersect F(i-1);
               Z := `minus`(S, Y);
               for s from 1 to i + p - 1 do
                    if not `in`(s, S) then
                         for m from 0 to i + p - 1 do
                              if g[m+s] = g[m] then
                                   break;
                              end if;
                         end do;
                         if m = i + p then
                              if s < i then
                                   Y:= `union`(Y, {s});
                              else
                                   Z:= `union`(Z, {s});
                              end if;
                         end if;
                    end if;
               end do;
               ###
               print(expansion set);
               return Y,Z^(p);
               ###
          end if;
          ######
          break;
     end if;
	 #########
end do;
############
end proc;
\end{verbatim}
\subsection{The subcode ``mex"}  \label{A2}
\begin{verbatim}
mex := proc (S)
local i;
if S = {} then
    return 0;
else
    for i from 0 to max(S)+1 do
        if not `in`(i, S) then
            return i;
        end if;
    end do;
end if;
end proc;
\end{verbatim}
\subsection{The subcode ``F"}  \label{A3}
\begin{verbatim}

F := proc (n)
local i, S;
S := {};
for i to n do
     S := `union`(S, {i});
end do;
return S;
end proc;
\end{verbatim}
\subsection{Two examples}  \label{A4}
\begin{verbatim}

sub({2, 7, 10}, 100);
                      pre_period length
                               0
                         period length
                               17
                    pre_period nim_values
                      periodic nim_values
       0, 0, 1, 1, 0, 0, 1, 1, 2, 0, 3, 1, 2, 0, 3, 1, 2
                         expansion set
                         {2, 7, 10}^17
sub({3, 8, 12}, 100);
                       pre_period length
                               16
                         period length
                                5
                   pre_period nim_values
         0, 0, 0, 1, 1, 1, 0, 0, 2, 1, 1, 0, 2, 2, 1, 3
                      periodic nim_values
                         0, 0, 2, 1, 1
                         expansion set
                     {3, 8, 12, 13}, {18}^5
\end{verbatim}


\section{Pre-periodic nim-values of some ultimately periodic games}  \label{B1}

\bigskip

\begin{center}
\begin{tabular}{|c|c|c|}\hline
\begin{tabular}{c}
Subtraction set \\
\{1,a,b\} \\
$a$: even\\
$b = a+r$\\
$r$: odd
\end{tabular}           & $n_0$    & pre-periodic nim-values     \\ \hline
$3 \leq r \leq a-3$     & $b2$       & $(01)^{\frac{a}{2}}2(01)^{\frac{r-1}{2}}23$ \\
\hline
\end{tabular}
\begin{table}[ht]
\caption{Pre-periodic nim-values of ultimately periodic games $\S(1,a,b)$ in Table \ref{T.2.3.1}.} \label{T.2.3.1.b}
\end{table}
\end{center}
\begin{center}
\begin{tabular}{|c|c|c|c|}\hline
\begin{tabular}{c}
Subtraction set \\
$\{1,a,b\}$  \\
$a$: even\\
$b = 2a+r$
\end{tabular}           & $r$  & $n_0$    & pre-periodic nim-values     \\ \hline
$2 = r < a-4$        & \multirow{4}{*}{even}
                               & $a+2b+4$   &
                               \begin{tabular}{c}
                         $(01)^{\frac{a}{2}} 2 (01)^{\frac{a}{2}}(23)^{\frac{a}{2}} 013$ \\
                               $(01)^{\frac{a}{2}} 2 (01)^{\frac{a}{2}-1}20123$
                               \end{tabular}
                                    \\ \cline{1-1} \cline{3-4}
$4 \leq r < a-4$

                        & &$5a+12$  & \begin{tabular}{c}
                        $((01)^{\frac{a}{2}} 2)^2 012 (32)^{\frac{a}{2}-2} (01)^22 $\\
                        $(01)^{\frac{a}{2}} 2 (01)^{\frac{a}{2}-1}2 (01)^2 23$
                        \end{tabular}
                                    \\ \cline{1-1} \cline{3-4}
$4 < r = a-4$

                        &      & $4a+b$  & \begin{tabular}{c}
                         $((01)^{\frac{a}{2}} 2)^2 (01)^{\frac{a}{2}-3} 2 (32)^2 (01)^{\frac{a}{2}-2}2 $ \\
                         $(01)^{\frac{a}{2}} 2 (01)^{\frac{a}{2}-1} 2 (01)^{\frac{a}{2}-2} 2 3$
                         \end{tabular}
                                    \\ \cline{1-1} \cline{3-4}

$4 \leq r = a-2$        &       & $b+2$  & $(01)^{\frac{a}{2}} 2 (01)^{\frac{a}{2}} 2 (01)^{\frac{a-4}{2}} 23$   \\    \hline
\end{tabular}
\begin{table}[ht]
\caption{Pre-periodic nim-values for ultimately periodic games $\S(1,a,b)$ in Table \ref{T.2.3.2}.} \label{T.2.3.2.b}
\end{table}
\end{center}

\begin{center}
\begin{tabular}{|c|c|c|c|}\hline
\begin{tabular}{c}
Subtraction set \\
\{1,a,3a+r\} \\
$a$: even
\end{tabular}
             & $r$    & $n_0$  & pre-periodic nim-values    \\ \hline
$3 = r < a-5$& \multirow{5}{*}{odd}
                      & $2a+3b+6$    & \begin{tabular}{c}
                             $((01)^{\frac{a}{2}}2)^3 (32)^{\frac{a}{2}-1}3      013 $ \\
                             $((01)^{\frac{a}{2}}2)^2 (01)^{\frac{a}{2}-1}2 012
                                                       (32)^{\frac{a}{2}-2} (01)^2 2$  \\
                             $(01)^{\frac{a}{2}}2   ((01)^{\frac{a}{2}-1}2)^2 (01)23$
                             \end{tabular}
                             \\ \cline{1-1} \cline{3-4}
$3 < r < a - 5$ &     & $10a+b-8$   & \begin{tabular}{c}
$((01)^{\frac{a}{2}}2)^3 (01)^{\frac{r-3}{2}} 2 (32)^{\frac{a}{2}-2} (01)^2 2 $ \\
$((01)^{\frac{a}{2}}2)^2 (01)^{\frac{a}{2}-1}2 (01)^2 2 (32)^{\frac{a}{2}-3} (01)^{\frac{a}{2}-3} 2$\\
  $(01)^{\frac{a}{2}}2   ((01)^{\frac{a}{2}-1}2)^2 (01)^{\frac{a}{2}-3}23$
                             \end{tabular}  \\ \cline{1-1} \cline{3-4}
$5 \leq r = a - 5$ &  & $10a+b-4$   & \begin{tabular}{c}
$((01)^{\frac{a}{2}}2)^3 (01)^{\frac{r-3}{2}} 2 (32)^{3} (01)^{\frac{a}{2}-3}2 $ \\
$((01)^{\frac{a}{2}}2)^2 (01)^{\frac{a}{2}-1}2 (01)^{\frac{a}{2}-3}2 (32)^2 (01)^{\frac{a}{2}-2} 2$\\
  $(01)^{\frac{a}{2}}2   ((01)^{\frac{a}{2}-1}2)^2 (01)^{\frac{a}{2}-2}23$
                             \end{tabular}  \\ \cline{1-1} \cline{3-4}
$5 \leq r = a - 3$ &         & $a+2b+4$   & \begin{tabular}{c}
         $((01)^{\frac{a}{2}} 2)^3 (01)^{\frac{r-3}{2}} 2 (32)^2 (01)^{\frac{r-1}{2}}2 $ \\ $((01)^{\frac{a}{2}} 2)^2 (01)^{\frac{a}{2}-1} 2 (01)^{\frac{a}{2}-2} 2 3$
         \end{tabular}  \\ \cline{1-1} \cline{3-4}
$5 \leq r = a-1$   &         & $b+2$    & $((01)^{\frac{a}{2}}2)^3 (01)^{\frac{a}{2}-2} 23$   \\ \hline
\end{tabular}
\begin{table}[ht]
\caption{Pre-periodic nim-values for ultimately periodic games $\S(1,a,b)$ in Table \ref{T.2.3.3}.} \label{T.2.3.3.b}
\end{table}
\end{center}

\begin{center}
\begin{tabular}{|c|c|c|}\hline
\begin{tabular}{c}
Subtraction set \\
$\{a,b,b+a+1\}$ \\
$a \geq 4$ \\
$b = 2a+r$ \\
$r \leq 2$
\end{tabular}
     & $n_0$  & pre-periodic nim-values   \\ \hline
$\{a,2a,3a+1\}$     & $4a$      & $0^a1^a2^a03^{a-1}$  \\ \hline
\{a,2a+2,3a+3\}     & $6a+6$     & $0^a1^a0^22^{a-2}1^203^{a-3}2^2120^{a-1}31^{a-1}0^23$    \\ \hline
\end{tabular}
\begin{table}[ht]
\caption{Pre-periodic nim-values for ultimately periodic games $\S(a,b,b+a+1)$ in Table \ref{T.2.6}.} \label{T.2.6.b}
\end{table}
\end{center}


\end{document}